\documentclass[12pt]{article}
\usepackage[utf8]{inputenc}
\usepackage[T1]{fontenc}
\usepackage{tgtermes}
\usepackage{fullpage}

\usepackage{amsmath}
\usepackage{amssymb}
\usepackage{amsthm}
\usepackage{amsfonts}
\usepackage{xcolor}
\usepackage{mathtools}
\usepackage{tikz}
\usetikzlibrary{arrows.meta}

\usepackage{ytableau}

\usepackage{hyperref}
\hypersetup{
    colorlinks = true,
    linkcolor = black,
    urlcolor = blue,
    citecolor = blue
    }
\usepackage{cleveref}

\newcommand{\NN}{\mathbb N}
\newcommand{\ZZ}{\mathbb Z}

\newcommand{\RR}{\mathbb R}

\DeclareMathOperator{\mO}{\mathcal{O}}
\DeclareMathOperator{\mC}{\mathcal{C}}
\DeclareMathOperator{\id}{id}
\DeclareMathOperator{\conv}{conv}

\DeclareMathOperator{\GT}{GT}
\DeclareMathOperator{\SL}{\mathfrak s \mathfrak l}
\DeclareMathOperator{\diags}{Diag}

\theoremstyle{plain}
\newtheorem{theorem}{Theorem}[section]
\newtheorem{lemma}[theorem]{Lemma}
\newtheorem{corollary}[theorem]{Corollary}
\newtheorem{prop}[theorem]{Proposition}

\newtheorem{question}[theorem]{Question}

\theoremstyle{definition}
\newtheorem{definition}[theorem]{Definition}
\newtheorem{remark}[theorem]{Remark}
\newtheorem*{notation}{Notation}
\newtheorem{example}[theorem]{Example}
\newtheorem{setup}[theorem]{Setup}

\title{Restricted Chain-Order Polytopes\\
via Combinatorial Mutations}
\author{Oliver Clarke, Akihiro Higashitani, and Francesca Zaffalon}
\date{October 2022}

\begin{document}

\maketitle
\begin{abstract}
    We study restricted chain-order polytopes associated to Young diagrams using combinatorial mutations.
    These polytopes are obtained by intersecting chain-order polytopes with certain hyperplanes. 
    The family of chain-order polytopes associated to a poset interpolate between the order and chain polytopes of the poset. Each such polytope retains properties of the order and chain polytope; for example its Ehrhart polynomial.
    For a fixed Young diagram, we show that all restricted chain-order polytopes are related by a sequence of combinatorial mutations.
    Since the property of giving rise to the period collapse phenomenon is invariant under combinatorial mutations, we provide a large class of rational polytopes that give rise to period collapse.
\end{abstract}

% \tableofcontents

\section{Introduction}

Let $P \subset \RR^N$ be a $d$-dimensional rational polytope. The lattice point counting function $L_P(n) := \lvert nP \cap \ZZ^N \rvert$ is a \emph{quasi-polynomial} in $n$ of degree $d$, that is, a polynomial $L_P(n) = c_d(n)n^d + c_{d-1}(n)n^{d-1} + \dots + c_1(n)n + c_0(n)$ whose coefficients $c_i(n)$ are periodic in $n$ \cite{beck2007computing}. The least common multiple of the periods of $c_i$ is called the \emph{period of $P$}. Generically, the period of $P$ is equal to its \emph{denominator}, which is the smallest positive integer $m$ such that all vertices of the $m$th dilate of $P$ lie in $\ZZ^N$. By a well-known theorem of Ehrhart, the period divides the denominator. So all lattice polytopes are polytopes with period one, however the converse is false. We say that $P$ has \emph{period collapse} if its period is not equal to its denominator. 
It was proved in \cite[Theorem 2.2]{MW_2005} that, for all positive integers $d$, $D$, and $s$ with $d \geq 2$ such that $s$ divides $D$, there exists a $d$-dimensional rational polytope with its denominator $D$ and period $s$. 
Since then, period collapse has become one of the main topics in Ehrhart theory. See, e.g, \cite{BSW_2008, HM_2008, MW_2005}. 

Our method to study polytopes involves \emph{combinatorial mutations} \cite{Akhtar2012, higashitani2020two}, which were originally defined in the study of mirror symmetry for Fano varieties. There are two complementary perspectives on combinatorial mutations that are related by taking the dual. Our perspective is derived from the so-called $M$-lattice. That is, we consider a combinatorial mutation to be a piece-wise linear map. In this setting, polytopes that are related by a combinatorial mutation have the same Ehrhart polynomial. See \cite[Proposition~4]{Akhtar2012} and \Cref{prop: mutation equiv same Ehrhart}. In particular, if $P$ has period collapse, then all polytopes that are \emph{mutation equivalent} to $P$ also have period collapse. 

Despite preserving their Ehrhart polynomials, many other salient features of polytopes are not invariant under combinatorial mutations. For example, the denominator and the number of vertices may change after a combinatorial mutation.

\begin{example}\label{example: intro rational to lattice mutation}
Let $P \subset \RR^2$ be the convex hull of the rational points $(1,0)$, $(0,-1)$, $(-1,0)$ and $(0, \frac 12)$. Let $\varphi : \RR^2 \rightarrow \RR^2$ be the piece-wise linear map given by
\[
\varphi(x,y) = \begin{cases}
(x,y) & \text{if } x \le 0, \\
(x, x+y) & \text{if } x \ge 0.
\end{cases}
\]
The map $\varphi$ is an example of a \emph{tropical map $\varphi_{w,F}$} with respect to the data $w = (0,1)$ and $F = \conv\{(0,0), (-1,0)\}$. See \Cref{sec: prelim combinatorial mutations}. The image of $P$ under $\varphi$ is the lattice simplex $Q$ with vertices $(-1,0)$, $(0,-1)$ and $(1,1)$. See \Cref{fig: intro mutation example}.
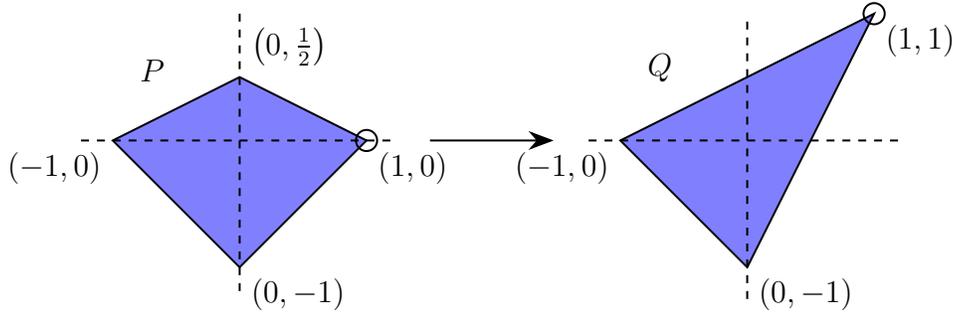
\begin{figure}[h]
    \centering
    % -- Tikz picture drawn in Matcha (modified) --
    \tikzset{every picture/.style={line width=0.75pt}}
    \begin{tikzpicture}[x = 0.6pt, y = 0.6pt, yscale = -1, xscale = 1] %invert y coordinates
    
    % Filled shapes:
    \draw [fill={blue!50}] 
    (160,220) node [anchor=north west] {$( 0,-1)$} -- 
    (240,140) node [anchor=north west] {$( 1,0)$} --
    (160,100) node [anchor=south west] {$\left( 0, \frac 12\right)$} -- 
    (80,140) node [anchor=north east] {$( -1,0)$} -- 
    cycle;
    \draw  [fill={blue!50}] 
    (400,140) node [anchor=north east] {$( -1,0)$} -- 
    (480,220) node [anchor=north west] {$( 0,-1)$} -- 
    (560,60) node [anchor=north west] {$( 1,1)$} -- 
    cycle;
    
    % Axes
    \draw [dashed]  (160,60) -- (160,240);
    \draw [dashed]  (480,240) -- (480,60);
    \draw [dashed]  (60,140) -- (260,140);
    \draw [dashed]  (380,140) -- (580,140);
    
    % Map arrow
    \draw [-{Stealth[scale=1.5]}] (280,140) -- (358,140);
    
    % Moved vertex
    \draw (240, 140) circle (4pt);
    \draw (560, 60) circle (4pt);
    
    % Text
    \draw (90,110) node [anchor=south west] {$P$};
    \draw (410,110) node [anchor=south west] {$Q$};
    \end{tikzpicture}
    \caption{The combinatorial mutation in \Cref{example: intro rational to lattice mutation}. The map $\varphi$ acts by the identity on the half-space $\{(x,y) \mid x \le 0\}$ and acts by a shear in the direction $(0,1)$ on $\{(x,y) \mid x \ge 0\}$. The circled vertex $(1,0) \in P$ is mapped to the circled vertex $(1,1) \in Q$ by $\varphi$.}
    \label{fig: intro mutation example}
\end{figure}
Since $P$ and $Q$ are mutation equivalent, they have the same Ehrhart polynomial. Moreover, the polytope $Q$ is a lattice polytope, so we deduce that $P$ has period collapse.
\end{example}

In this paper we focus on a class of polytopes related to poset, Gelfand-Tsetlin, and Birkhoff polytopes \cite{andersson2022restricted, ardilla2011GTandFFLVPolytopes, Gelfand1950finite, MR824105}. The Gelfand-Tsetlin polytopes $\GT_{\alpha, \beta}$ arise from representation theory, where their lattice points form a basis for the corresponding irreducible representation of the Lie algebra $\SL_n$. Using the perspective from representation theory, it is known that the Ehrhart quasi-polynomials of such Gelfand-Tsetlin polytopes are indeed polynomials \cite{kirillov2001ubiquity}. The Gelfand-Tsetlin polytope is an example of the \emph{order polytope} of a poset intersected with certain affine-hyperplanes. The order and chain polytopes, originally defined by Stanley \cite{MR824105}, are polytopes associated to a poset. The chain-order polytopes are polytopes that interpolate between the order and chain polytopes via a sequence of piece-wise linear maps called \emph{transfer maps} \cite[Definition 3.1]{MR824105}. It was shown that these piece-wise linear maps, similarly to combinatorial mutations, preserve the Ehrhart polynomial. Note that the chain-order polytopes were originally defined for \emph{marked} posets (see \cite{FF_2016}), which is a vast generalisation. In this paper, we specialise this definition so that our notions are concurrent with usual poset polytopes.

Recently, in \cite{andersson2022restricted}, it was shown that the restricted order and restricted chain polytopes that correspond to certain Gelfand-Tsetlin and Birkhoff polytopes, respectively, are related by a sequence of piece-wise linear maps. These piece-wise linear maps were originally defined by Pak \cite{pak2001hook}, and is almost equivalent to the Robinson–Schensted–Knuth (RSK) correspondence \cite{schensted_1961}. However, as noted in \cite[Section~5]{andersson2022restricted}, the transfer maps do not commute with taking intersections with the hyperplanes that define the restricted order polytope and restricted chain polytope. In this paper we resolve this problem by decomposing Pak's piece-wise linear map into smaller pieces, which are combinatorial mutations. See \Cref{proposition: Pak map as tropical maps}. We define the restricted analogues of chain-order polytopes in \Cref{def: restricted chain-order polytope} and prove that these polytopes are mutation equivalent, providing a complete answer to \cite[Remark~6.4]{andersson2022restricted}.

\medskip
\noindent \textbf{Outline.}
In \Cref{sec: preliminaries} we introduce the necessary preliminaries to state our main results. In particular, we define combinatorial mutations in \Cref{sec: prelim combinatorial mutations}, recall the definitions of polytopes from posets and Young diagrams in \Cref{sec: prelim posets and polytopes}, and fix our main setup (\Cref{setup: restricted chain-order polytope}) in \Cref{sec: prelim restricted chain-order polytopes}. In \Cref{sec: main results}, we describe our main results. Our main theorem (\Cref{thm : restricted order and chain mutation equivalent}) shows mutation equivalence of the restricted order and chain polytopes. We describe a sequence of mutations that relate these polytopes and define our main tool for proving \Cref{thm : restricted order and chain mutation equivalent}: the \emph{restricted chain-order polytopes} in \Cref{def: restricted chain-order polytope}. In \Cref{sec: proofs}, we prove our main results in two steps using Lemmas~\ref{lemma : chi_r bijection restricted polytopes} and \ref{lemma: tropical map is a mutation}. In \Cref{sec: computations} we provide some corollaries of our results, together with some computations and further questions. 

\medskip

\noindent{\bf Acknowledgements.}~The authors would like to thank Sam Hopkins for pointing out possible connections to their work. O.C. is an overseas researcher under Postdoctoral Fellowship of Japan Society for the Promotion of Science (JSPS). A.H. is partially supported by JSPS Fostering Joint International Research (B) 21KK0043 and JSPS Grant-in-Aid for Scientific Research (C) 20K03513. F.Z. is partially supported by grants G0F5921N (Odysseus programme), G023721N from the Research Foundation - Flanders (FWO) and FWO fundamental research fellowship (1189923N). 

\section{Preliminaries}\label{sec: preliminaries}

In this section, we recall the main definitions and fix our notation. See \Cref{setup: restricted chain-order polytope}. We use this notation to explain our main results in \Cref{sec: main results}.

\subsection{Combinatorial mutations}\label{sec: prelim combinatorial mutations}

Fix a natural number $N$. For each $x$ and $y$ in $\RR^N$, denote by $x \cdot y$ the usual dot product of $x$ and $y$. We recall the definition of combinatorial mutation of a polytope $P \subseteq \RR^N$ in terms of a piece-wise linear map. Let $w \in \ZZ^N$ be a primitive lattice point and $F \subset w^\perp \subset \RR^N$ a lattice polytope, where $w^\perp=\{x \in \RR^N \mid x \cdot w =0\}$. We define the \emph{tropical map}
\[ 
\varphi_{w,F} : \RR^N \to \RR^N, \quad x \mapsto x-x_{\min} w, 
\]
where $x_{\min} = \min\{ x \cdot f \mid f \in F\}$. Each vertex $v \in F$ defines a \textit{region of linearity} of $\varphi_{w, F}$ given by $U = \{x \in \RR^N \mid x_{\min} = x \cdot v \}$. The restriction $\varphi_{w,F}|_U : x \mapsto x - (x \cdot v)w$ is a unimodular map given by a shear.

Let $P \subset \RR^N$ be a polytope. If $\varphi_{w,F}(P)$ is convex, we say that $\varphi_{w,F}(P)$ is a \emph{combinatorial mutation} of $P$. Two polytopes $P$ and $Q$ in $\RR^N$ are said to be \emph{combinatorial-mutation equivalent}, or simply \emph{mutation equivalent}, if there exists a sequence of combinatorial mutations
\[
P_1 = \varphi_{w_1, F_1}(P_0), \ 
P_2 = \varphi_{w_2, F_2}(P_1), \dots , 
P_k = \varphi_{w_i, F_i}(P_{k-1}) \]
such that $P_0 = P$ and $P_k = Q$. We refer to the polytopes $P_2, P_3, \dots, P_{k-1}$ as the \emph{intermediate polytopes} of the sequence of combinatorial mutations.

\begin{remark}
Combinatorial mutations arise in the context of mirror symmetry in the study of the classification of Fano varieties \cite{Akhtar2012}. They have also been shown to connect families of Newton-Okounkov bodies for partial flag varieties \cite{clarke2021mutationsBlockDiagonal, clarke2022mutationsGTandFFLV, clarke2022toric} and adjacent tropical cones \cite{escobar2019wall}.
\end{remark}

The \emph{Ehrhart series} is the generating function of the number of lattice points $\lvert nP \cap \ZZ^N \rvert$ in the $n$th dilate of $P$. We denote it
\[
E_P(t) := \sum_{n \ge 0} \ \lvert nP \cap \ZZ^N \rvert \ t^n.
\]

\begin{prop}[{\cite[Proposition~4]{Akhtar2012}}]\label{prop: mutation equiv same Ehrhart}
Mutation equivalent polytopes have the same Ehrhart series.
\end{prop}

For clarity we provide our own proof of this proposition. Note that we do not assume that the dual polytopes are lattice polytopes.

\begin{proof}
Suppose that $\varphi_{w,F}(P)$ is a combinatorial mutation of $P$. The regions of linearity of $\varphi_{w,F}$ are the maximal cones of a polyhedral fan $\Sigma$. Let $\Sigma^\circ = \{\sigma^\circ \mid \sigma \in \Sigma\}$ be the collection of relative interiors of cones in $\Sigma$. Note that $\Sigma$ is a complete fan, so $\RR^N = \bigcup_{\sigma^\circ \in \Sigma} \sigma^\circ$. Since $\varphi_{w,F}$ is piece-wise unimodular, for each $\sigma^\circ \in \Sigma^\circ$, we have that $\varphi_{w,F}(P \cap \sigma^\circ)$ and $P \cap \sigma^\circ$ have the same Ehrhart series. So, we have
\[
E_{\varphi_{w,F}(P)}(t) = 
\sum_{\sigma^\circ \in \Sigma^\circ} E_{\varphi_{w,F}(P \cap \sigma^\circ)}(t) =
\sum_{\sigma^\circ \in \Sigma^\circ} E_{P \cap \sigma^\circ}(t) =
E_{P}(t).
\]
\end{proof}

\subsection{Posets and polytopes}\label{sec: prelim posets and polytopes}

\newcommand{\posetSymbol}{\mathcal P}
Let $(\posetSymbol, <)$ be a partially ordered set, which we usually denote as $\posetSymbol$. A subset $U \subseteq \posetSymbol$ is called an \textit{up-set} if for all $x \in U$ and $y \in \posetSymbol$, we have that $x < y$ implies that $y \in U$. A subset $D \subseteq \posetSymbol$ is a \textit{down-set} if its complement $\posetSymbol \backslash D$ is an up-set. A subset $S \subseteq \posetSymbol$ is called a \textit{chain} if each pair of elements in $S$ is comparable. A subset $S \subseteq \posetSymbol$ is an \textit{antichain} if no pair of elements in $S$ is comparable. Given a subset $S \subseteq \posetSymbol$, the collections of minimal and maximal elements are respectively $\min(S) = \{s \in S \mid x \nless s \text{ for all } x \in S\}$ and $\max(S) = \{s \in S \mid x \ngtr s \text{ for all } x \in S \}$.

We denote by $\NN = \{1,2,\dots\}$ the set of natural numbers and the set $[n] = \{ 1, 2, \dots, n\}$ for any $n \in \NN$. We equip $\NN^2$ with the component-wise partial order: $(a,b) \le (c,d)$ if $a \le c$ and $b \le d$. A \textit{Young diagram} $\lambda \subseteq \NN^2$ is a finite down-set, which we take to be a sub-poset of $\NN^2$. Young diagrams are often defined as partitions of natural numbers. Explicitly, a partition $(\lambda_1 \ge \lambda_2 \ge \dots \ge \lambda_k)$ of $n \in \NN$ is naturally associated to the down-set 
\[
\lambda = \{(1,1),(1,2), \dots, (1,\lambda_1), \ (2,1), (2,2), \dots, (2,\lambda_2), \dots \ , (k,1), \dots , (k, \lambda_k)\} \subseteq \NN^2.
\]
Typically, a Young diagram $\lambda$ is depicted as a collection of boxes, with $(1,1)$ located in the top-left, $(1,2)$ to the immediate right of $(1,1)$, $(2,1)$ immediately below $(1,1)$, and so on. We say that a box $r \in \lambda$ is a \emph{corner} if $r \in \max(\lambda)$ is a maximal element. 
% A Young diagram $\lambda$ with the partial order induced by $\NN^2$ defines a poset.

\begin{example}\label{example : Young diagram}
The Young diagram given by the partition $(4,4,3)$, is represented via the following collection of boxes.
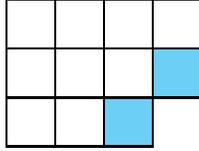
\[
\ydiagram [*(cyan!50)]{0,3+1,2+1}
*[*(white) ]{4,4,3}
\]
The Young diagram has two corners: $(2,4)$ and $(3,3)$, which are the shaded boxes above.
\end{example}

We write $\RR^\lambda$ for the real vector space 
with distinguished basis $\{e_p \mid p \in \lambda\}$. Given $x, y \in \RR^\lambda$ we write $x \cdot y = \sum_{p \in \lambda} x_p y_p$ for the standard dot-product. For all $x \in \RR^\lambda$ and $p \notin \lambda$, we take the convention that $x_p = 0 \in \RR$ and $e_p = 0 \in \RR^\lambda$. 

We recall the definition of two polytopes classically associated to a poset. 
% \begin{definition}
% Let $\lambda$ be a Young diagram. 
The \emph{order polytope} of the Young diagram $\lambda$ is the polytope
% \[ 
% \mO(\lambda) = \left\{ x \in \RR^{\lambda} \left| \  
% \begin{matrix*}[l]
% x_{(i_1,j_1)} \leq x_{(i_2,j_2)} & \text{ for all } (i_1,j_1)\leq (i_2,j_2) \text{ in } \lambda, \\ 
% 0 \leq x_{(i,j)} \leq 1  & \text{ for all } (i,j) \in \lambda \end{matrix*} \right. \right\}. 
% \]
\[ 
\mO(\lambda) = \left\{ x \in \RR^{\lambda} \mid  
0 \le x_p \le x_q  \le 1 \text{ for all } p\leq q \text{ in } \lambda \right\}. 
\]
The \emph{chain polytope} of $\lambda$ is the polytope
% \[ 
% \mC(\lambda) = \left\{ x \in \RR^{\lambda} \left| \ \begin{matrix*}[l] 0 \leq x_{(i,j)} & \text{ for all } (i,j) \in \lambda, \\
% x_{(i_1,j_1)}+x_{(i_2,j_2)} + \dots + x_{(i_k,j_k)} \leq 1 & \text{ for all } (i_1,j_1)< (i_2,j_2) < \dots < (i_k,j_k) \text{ in } \lambda \end{matrix*}\right. \right\}. 
% \]
\[ 
\mC(\lambda) = \left\{ x \in \RR^{\lambda} \mid
0 \le x_{p_1}+x_{p_2} + \dots + x_{p_k} \le 1 \text{ for all } p_1 < p_2 < \dots < p_k \text{ in } \lambda \right\}. 
\]
% \end{definition}

Each point $x \in \RR_{\ge 0}^\lambda$ can be thought of as a non-negative \textit{filling} of the Young diagram, i.e.~writing the value $x_p$ in the box $p \in \lambda$. A point $x \in \RR_{\ge 0}^\lambda$ lies in the $k$th dilate $k\mO(\lambda)$ if and only if the values in each box do not exceed $k$ and increase when moving down and to the right. Similarly, a point $x \in \RR_{\ge 0}^\lambda$ lies in $k\mC(\lambda)$ if and only if the sum of values along any path in $x$ that starts at $(1,1)$ and moves down and to the right is at most $k$. The vertices of these polytopes have the following descriptions.

\begin{prop}[{\cite[Corollary~1.3 and Theorem~2.2]{MR824105}}]
Fix a Young diagram $\lambda$ and for each subset $S \subseteq \lambda$ define the characteristic vector $\chi(S) \in \RR^\lambda$ by $\chi(S)_p = 1$ if $p \in S$ and $\chi(S)_p = 0$ if $p \notin S$. The vertices of $\mO(\lambda)$ and $\mC(\lambda)$ are \[
V(\mO(\lambda)) = \{\chi(S) \mid S \text{ an up-set of } \lambda \}
\quad \text{ and } \quad
V(\mC(\lambda)) = \{\chi(S) \mid S \text{ an antichain of } \lambda \}.
\]
\end{prop}

The order polytope and chain polytope are mutation equivalent \cite{higashitani2020two}. The sequence of mutations can be realised by a decomposition of a piece-wise linear map called the \textit{transfer map}, introduced in \cite{MR824105}, which interpolates between the polytopes. Intermediate polytopes of this sequence of mutations are given by the so-called chain-order polytope.

\begin{definition}
Let $\lambda$ be a Young diagram and let $C \subsetneq \lambda$ be a proper up-set. The \emph{chain-order polytope} of $\lambda$ with respect to $C$ is the polytope
\[ 
\mO_C(\lambda) = \left\{ x \in \RR^\lambda \left| \ \begin{matrix*}[l] 
0 \leq x_p \leq 1 & \text{ for all } p \in \lambda,\\
x_p \leq x_q & \text{ for all } p \leq q \text{ in } \lambda \backslash C, \\
x_p+ x_{q_1} + \dots + x_{q_n} \leq 1 & \text{ for all } p \in \lambda \backslash C \text{ and } q_1, \dots, q_n \in C \\
& \qquad \text{ such that } p < q_1 < \dots < q_n. \end{matrix*}  \right. \right\}.
\]
If $C = \emptyset$, then $\mO_C(\lambda) = \mO(\lambda)$ coincides with the order polytope. If $C= \lambda \backslash \{(1,1)\}$, then $\mO_C(\lambda) = \mC(\lambda)$ coincides with the chain polytope. For $C = \lambda$, we define $\mO_C(\lambda) = \mC(\lambda)$ to be the chain polytope. %\textcolor{gray}{(Why do we say that we define the chain polytope in this way? We already gave the definition of chain polytope)
%Ollie's answer. Because, in the definition above, if we take $C = \lambda$ to be the entire poset, then there are no inequalities of the third type, since $\lambda \backslash C$ is empty. However, we want $\mO_C(\lambda)$ to be equal to the chain polytope when $C = \lambda$. Alternatively, we could define a quantity $C_{\max}(x; q)$ which is equal to the maximum of $0$ and all $x_p$ where $p \in \lambda \backslash C$ and $p < q$ in $\lambda$. Then the third inequality is $C_{\max}(x; q) + x_{q_1} + \dots + x_{q_n} \le 1$ for all \dots. Maybe there's another way? Like saying $p \in (\lambda \backslash C) \cup \{\infty\}$ where $\infty$ is just some element not in $\lambda$, and then use the convention that $x_{\infty} = 0$. I'm not too sure what's best.}
\end{definition}

\begin{example}
Consider the Young diagram $\lambda$ given by the partition $(3,2)$ and the up-set $C = \{(1,2),(1,3), (2,2)\}$ corresponding to the highlighted boxes below. The chain-order polytope of $\lambda$ with respect to $C$ has vertices:
\[ \begin{ytableau} 0 & *(yellow!50) 0 & *(yellow!50) 0 \\
0 & *(yellow!50) 0  \end{ytableau}\, , \;
\begin{ytableau} 0 & *(yellow!50) 0 & *(yellow!50) 0 \\
1 & *(yellow!50) 0  \end{ytableau}\, , \;
\begin{ytableau} 1 & *(yellow!50) 0 & *(yellow!50) 0 \\
1 & *(yellow!50) 0  \end{ytableau}\, , \;
\begin{ytableau} 0 & *(yellow!50) 1 & *(yellow!50) 0 \\
0 & *(yellow!50) 0  \end{ytableau}\, , \;
\begin{ytableau} 0 & *(yellow!50) 1 & *(yellow!50) 0 \\
1 & *(yellow!50) 0  \end{ytableau}\, ,\]
\[\begin{ytableau} 0 & *(yellow!50) 0 & *(yellow!50) 0 \\
0 & *(yellow!50) 1  \end{ytableau}\, , \;
\begin{ytableau} 0 & *(yellow!50) 0 & *(yellow!50) 1 \\
0 & *(yellow!50) 0  \end{ytableau}\, , \;
\begin{ytableau} 0 & *(yellow!50) 0 & *(yellow!50) 1 \\
1 & *(yellow!50) 0  \end{ytableau}\, \text{ and } \;
\begin{ytableau} 0 & *(yellow!50) 0 & *(yellow!50) 1 \\
0 & *(yellow!50) 1  \end{ytableau}\, .
\]
\end{example}

\subsection{Restricted chain and order polytopes}\label{sec: prelim restricted chain-order polytopes}

We now consider restricted versions of the order and chain polytopes; that is, the intersection of these polytopes with certain hyperplanes. To define them, we fix the following setup that will be used throughout the rest of the paper. 

\begin{setup}\label{setup: restricted chain-order polytope}
Fix a Young diagram $\lambda$. Let $m_1,m_2 \in \NN$ be smallest natural numbers such that $\lambda \subseteq [m_1]\times [m_2]$. For example, if $\lambda$ is given by the partition $\{\lambda_1 \ge \lambda_2 \ge \dots \ge \lambda_s \}$, then $(m_1,m_2) = (s,\lambda_1)$. Fix a vector $\underline{d} = (d_{m_1 - 1}, d_{m_1 -2}, \dots, d_0, d_{-1}, \dots, d_{1-m_2}) \in \NN^{m_1+m_2-1}$ and a natural number $k \in \NN$. Define the set of integers $\diags(\lambda) := \{1-m_2, 2-m_2, \dots, m_1-1\}$, which index the \emph{diagonals} of $\lambda$. For each $\ell \in \diags(\lambda)$, denote by $r_\ell$ the maximal element in the \emph{$\ell$th diagonal of $\lambda$}: $\{(i,j) \in \lambda \mid i-j = \ell\}$. Visually, the box $r_\ell$ is the bottom-right-most box of the Young diagram on the $\ell$th diagonal. See \Cref{fig: diagonals and r boxes}.
\begin{figure}[h]
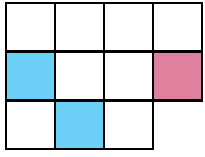

    \centering
    \[ \ydiagram [*(cyan!50)]{0,1,1+1}
    *[*(purple!50)]{0,3+1,0}
    *[*(white) ]{4,4,3}
    \]
    \caption{Highlighted in blue, the $1$st diagonal given by the boxes $(2,1)$ and $r_{1}=(3,2)$. Highlighted in purple, the box $r_{-2} = (2,4)$.}
    \label{fig: diagonals and r boxes}
\end{figure}
\end{setup}

\begin{definition}
Fix Setup~\ref{setup: restricted chain-order polytope}. The \emph{restricted order polytope of $\lambda$} with respect to $\underline{d}$ and $k$ is the restriction of the $k$th dilate of the order polytope of $\lambda$ defined by:
\[ 
\mO(\lambda)^k_{\underline{d}} = 
\left\{ x \in k \mO(\lambda) \left| \ \sum_{i-j = \ell} x_{(i,j)} = d_\ell \text{ for all } \ell \in \diags(\lambda) \right.\right\}. 
\]
The \emph{restricted chain polytope of $\lambda$} with respect to $\underline{d}$ and $k$ is the restriction of the $k$th dilate of the chain polytope of $\lambda$ defined by:
\[ 
\mC(\lambda)^k_{\underline{d}} = 
\left\{ x \in k \mC(\lambda) \left| \ \sum_{p\leq r_\ell} x_p = d_\ell \text{ for all } \ell \in \diags(\lambda) \right.\right\}. \]
\end{definition}

\begin{remark}
If $\lambda = \{(1,1), \dots, (n,n)\}$ is square and $\underline{d}=(1,2,\dots,n-1,n,n-1,\dots,1)$, then the restricted order polytope is equal to the restricted Gelfand-Tsetlin polytope and the restricted chain polytope is equal to the restricted Birkhoff polytope. The period-collapse phenomenon of these polytopes is studied in \cite{andersson2022restricted}. In \Cref{sec: computations}, we show that period-collapse also occurs for some other families of polytopes.
\end{remark} 

\section{Main results}\label{sec: main results}

In this section we explain our main results, which use combinatorial mutations to connect restricted order polytopes to restricted chain polytopes. In particular, by \Cref{prop: mutation equiv same Ehrhart}, if a polytope exhibits period collapse, then all mutation equivalent polytopes simultaneously exhibit period collapse.

\begin{theorem}\label{thm : restricted order and chain mutation equivalent}
Fix \Cref{setup: restricted chain-order polytope}. The restricted order polytope $\mO(\lambda)^k_{\underline{d}}$ and the restricted chain polytope $\mC(\lambda)^k_{\underline{d}}$ are mutation equivalent.
\end{theorem}

We construct a sequence of mutation equivalent polytopes that interpolate between the restricted order and restricted chain polytopes. We do this by defining a restricted analogue of the chain-order polytope. We show that all piece-wise linear maps that connect these polytopes are combinatorial mutations. The foundation for our construction is the piece-wise linear map introduced by Pak \cite{pak2001hook}.

\medskip

\noindent \textbf{Pak's piece-wise linear map.}
% which sends $\mO(\lambda)^k_{\underline{d}}$ to $\mC(\lambda)^k_{\underline{d}}$. We begin by recalling the definition of this map. 
Fix \Cref{setup: restricted chain-order polytope}. We recall, from \cite[Section~4]{pak2001hook}, the piece-wise linear map $\xi_\lambda : \RR^\lambda \rightarrow \RR^\lambda$ which bjectively maps $\mO(\lambda)^k_{\underline{d}}$ to $\mC(\lambda)^k_{\underline{d}}$. The map is defined inductively on $|\lambda|$. If $|\lambda| =1$, then $\xi_\lambda = \id$. Otherwise, if $|\lambda|>1$, then let $r=r_{\ell}$ be a corner of $\lambda$. We recall the convention that for any $x \in \RR^\lambda$, we have $x_p = 0$ if $p \notin \lambda$.
Define the map $\chi_{r}: \RR^\lambda \to \RR^\lambda$ as follows:
\[ 
\chi_{r}(x)_{(i,j)} = 
\begin{cases} x_{(i,j)} & \text{ if } i-j \neq \ell, \\
x_{(i,j)}-\max\{ x_{(i-1,j)},x_{(i,j-1)} \} & \text{ if } (i,j) = r,\\
\max\{x_{(i-1,j)},x_{(i,j-1)}\} & \\
\quad + \min\{ x_{(i+1,j)}, x_{(i,j+1)} \}-x_{(i,j)} &\text{ if } i-j = \ell \text{ and } (i,j) \neq r.
\end{cases}
\]
By induction, the map $\xi_{\lambda \backslash \{r\}}$ is already defined, so we define $\xi_\lambda = (\xi_{\lambda \backslash \{r\}} \times \id) \circ \chi_{r}$. By \cite[Theorem~4]{pak2001hook}, the map $\xi_\lambda$ is well-defined, in particular it does not depend on the choice of corner $r$ above.

We naturally extend the definition of the map $\chi_r$ to non-corners $r \in \lambda$. This is done by defining $\chi_r(x)_{(i,j)} = x_{(i,j)}$ for all $(i,j) \nless r$ and taking the above definition for all other coordinates $(i,j) \le r$. So we may express Pak's map as a composition: 
\begin{equation}\label{eqn: pak map decomp coarse}
\xi_\lambda = \chi_{r_1} \circ \chi_{r_2} \circ \dots \circ \chi_{r_n},    
\end{equation}
for any permutation $(r_1, \dots, r_n)$ of the elements of $\lambda$ that satisfies $r_i < r_j \implies i < j$.

\medskip \noindent
\textbf{A decomposition of Pak's map.}
Fix $r = (a, b) \in \lambda$ in the $\ell$th diagonal of $\lambda$.
We express the map $\chi_r$ as the composition of tropical maps and a unimodular map. In fact, we will show that this gives a sequence of combinatorial mutations. We define the tropical maps and unimodular map as follows:
\begin{itemize}
    \item For each $0 \le i \le i_{\max}(r) := \min\{a, b\}-1$, 
    let 
    \[
    w_i = e_{(a-i-1,\, b-i-1)} -e_{(a-i,\, b-i)} 
    \text{ and }
    F_i = \conv\left\{e_{(a-i-1,\, b-i)},\, e_{(a-i,\, b-i-1)}\right\}.
    \]
    Define the tropical map $\varphi_i := \varphi_{w_i,F_i}$.
    \item Let $\psi$ be the unimodular map:
    \[ 
    \psi(x)_{(i,j)} = 
    \begin{cases} x_{(i,j)} & \text{ if } i-j \neq \ell, \\
    x_{(i,j)} - (x_{i-1,j} + x_{i,j-1}) & \text{ if } (i,j) = r,\\
    -x_{(i,j)} + (x_{i-1,j} + x_{i,j-1}) & \text{ if } i-j = \ell \text{ and } (i,j) \neq r. \end{cases}
    \]
\end{itemize}

\begin{prop} \label{proposition: Pak map as tropical maps} With the notation above,
$\chi_r= \psi \circ \varphi_{i_{\max}(r)}\circ \dots \circ \varphi_0$.
\end{prop}
\begin{proof}
It is straightforward to show that the composition of the tropical maps is
\[ 
(\varphi_{i_{\max}(r)}\circ \dots \circ \varphi_0(x))_{(i,j)} = \begin{cases} x_{(i,j)} & \text{ if } i-j \neq \ell,\\
x_{(i,j)}+\min\{x_{(i-1,j)},x_{(i,j-1)}\} & \text{ if } (i,j)=r, \\
x_{(i,j)}+\min\{x_{(i-1,j)},x_{(i,j-1)}\} & \\
\quad -\min\{x_{(i+1,j)},x_{(i,j+1)}\} & \text{ if } i-j =\ell \text{ and } (i,j) < r. \end{cases}\]
For all $u, v \in \RR$ we have $\min\{u,v\}-u-v= -\max\{u,v\}$, so it follows that
\[ \psi \circ \varphi_{i_{\max}(r)}\circ \dots \circ \varphi_0(x) = \chi_r(x) \]
for every $x \in \RR^\lambda$.
\end{proof}

To prove that these tropical maps are combinatorial mutations, we proceed inductively. For each $0 \le i \le i_{\max}(r)$, we show that the image of the polytope under $\varphi_i$ is convex. As a result, we obtain a collection of mutation equivalent polytopes. Recall that Pak's map $\xi_\lambda$ can be expressed as a composition of maps $\chi_r$, as in \Cref{eqn: pak map decomp coarse}. Then, for each $1 \le s \le n$, the polytope 
\[
\mO_C(\lambda)_{\underline d}^k = 
\chi_{r_s} \circ \chi_{r_{s+1}} \circ \dots \circ \chi_{r_n}(\mO(\lambda)_{\underline d}^k)
\]
can be regarded as the restricted analogue of the chain-order polytope $\mO_C(\lambda)$ with respect to the up-set $C = \{r_s, r_{s+1}, \dots,  r_n\}$. We give an abstract definition of the polytope $\mO_C(\lambda)_{\underline d}^k$ and prove that the above equality holds in \Cref{sec: proofs}.

\begin{definition}\label{def: restricted chain-order polytope}
Fix Setup~\ref{setup: restricted chain-order polytope} and let $C \subseteq \lambda$ be an up-set. For each $\ell \in \diags(\lambda)$, we define:
\begin{itemize}
    \item $R_\ell=\{(a,b) \in \lambda \mid (a,b) \leq r_\ell\}$ the rectangular sub-Young diagram of $\lambda$,
    
    \item $S_\ell = \max((\lambda \backslash C) \cap R_\ell)$ the corners of $(\lambda \backslash C) \cap R_\ell$,
    
    \item $\overline{S_\ell} = \{(a,b) \in R_\ell \mid (a+i, b+i) \in S_\ell \text{ for some } i \ge 0 \}$ the diagonals of $R_\ell$ ending in $S_\ell$,
    
    \item $T_\ell = \min(C \cap R_\ell)$ the minimal elements of $C \cap R_\ell$,
    
    \item $\overline{T_\ell} = \{(a,b) \in R_\ell \mid (a+i, b+i) \in T_\ell \text{ for some } i \ge 1\}$ the diagonals in $R_\ell$ that end one step before $T_\ell$.
    
\end{itemize}
The \emph{restricted chain-order polytope} with respect to $\underline d \in \NN^{m_1 + m_2 - 1}$ and $k \in \NN$ is
% \[ 
% \mO_C(\lambda)^k_{\underline{d}} = 
% \left\{ x \in k \mO_C(\lambda) 
% \left| 
% \sum_{\substack{r_a \text{ corner of }\\ (\lambda\backslash C) \cap R_\ell}} 
% \sum_i x_{(i,i-a)} 
% - \sum_{\substack{r_b \text{ corner of }\\ (C\cap R_\ell)^*}} 
% \sum_i x_{i,i-b} 
% + \sum_{s \in C \cap R_\ell} x_s = d_\ell 
% \text{ for all } \ell  \right. \right\} 
% \]
\[ 
\mO_C(\lambda)^k_{\underline{d}} = 
\left\{ x \in k \mO_C(\lambda) 
\left| \ 
\sum_{s \in \overline{S_\ell}} x_s
- \sum_{t \in \overline{T_\ell}} x_t
+ \sum_{u \in C \cap R_\ell} x_u = d_\ell 
\text{ for each } \ell \in \diags(\lambda) \right. \right\}\, . 
\]
\end{definition}

\begin{example}\label{example: restricted chain order polytope}
Let $\lambda = \{(1,1), \dots, (5,6)\}$ be the rectangular Young diagram with $4$ rows and $5$ columns. Let $C \subset \lambda$ be the up-set with minimal elements $(4,3)$ and $(3,5)$. Fix $\ell = -1$, then $r_\ell=(5,6)$ and the set $R_\ell = \lambda$ is the entire poset. The sets $S_\ell$, $\overline S_\ell$, $T_\ell$ and $\overline T_\ell$ are shown in \Cref{fig: restricted chain order example}.

\begin{figure}[h]
    \centering
    \[
    \begin{ytableau}{}
 & \overline{s} & \overline{t} & & \overline{s} & \\
 \overline{t} & & \overline{s} & \overline{t} & & s \\
 & \overline{t} & & s & *(yellow!50) t & *(yellow!50) \\
 \overline{s} & & *(yellow!50) t & *(yellow!50) & *(yellow!50) & *(yellow!50) \\
 & s & *(yellow!50)  & *(yellow!50) & *(yellow!50) & *(yellow!50)
    \end{ytableau}
    \]
    \caption{Depiction of the Young diagram in \Cref{example: restricted chain order polytope}. Shaded boxes represent the up-set $C$. Fix $\ell = -1$. The set $S_\ell$ is given by the boxes labelled with $s$, the set $\overline S_\ell$ are the boxes $s$ and $\overline s$, the set $T_\ell$ are the boxes $t$, and the set $\overline T_\ell$ are the boxes $\overline t$.}
    \label{fig: restricted chain order example}
\end{figure}

\end{example}

\section{Proofs}\label{sec: proofs}

% \subsection{Combinatorial mutations and the proof of \texorpdfstring{\Cref{thm : restricted order and chain mutation equivalent}}{Theorem 4.6}}

In this section we prove all the results from \Cref{sec: main results}, in particular \Cref{thm : restricted order and chain mutation equivalent}, which follow from Lemmas~\ref{lemma : chi_r bijection restricted polytopes} and \ref{lemma: tropical map is a mutation}. We begin by fixing the notation for the affine hyperplanes that define the restricted chain-order polytopes. We then proceed to prove the two lemmas followed by the main theorem.

\begin{notation}
Following the notation from Definition~\ref{def: restricted chain-order polytope}, given a restricted chain-order polytope $\mO_C(\lambda)_{\underline d}^k$ and $\ell \in \diags(\lambda)$, we define the hyperplane \[
H_\ell = \left\{x \in \RR^\lambda \left\vert \ \sum_{s \in \overline{S_\ell}} x_s
- \sum_{t \in \overline{T_\ell}} x_t
+ \sum_{u \in C \cap R_\ell} x_u = d_\ell \right.\right\}.
\]
In particular we have that $\mO_C(\lambda)_{\underline d}^k = k\mO_C(\lambda) \cap \bigcap_\ell H_\ell$, where the intersection runs over $\diags(\lambda)$. Whenever we work with hyperplanes $H_\ell$ with respect to different up-sets $C$, we write $H_\ell^C$ for $H_\ell$ to avoid ambiguity. Similarly, we write $S_\ell^C$, $\overline{S_\ell^C}$, $T_\ell^C$ and $\overline{T_\ell^C}$ for the sets in Definition~\ref{def: restricted chain-order polytope}.
\end{notation}

\begin{lemma}\label{lemma : chi_r bijection restricted polytopes}
Fix \Cref{setup: restricted chain-order polytope} and let $C \subsetneq \lambda$ be an up-set. Let $r=(a,b)$ be a corner of $\lambda \backslash C$ and $\ell = a-b$. We denote by $\chi_r:\RR^\lambda \to \RR^\lambda$ the extension of Pak's map defined on $\RR^{\lambda\backslash C}$. Then $\chi_r: \mO_C(\lambda) \to \mO_{C\cup\{r\}}(\lambda)$ is a bijection. Recall the decomposition $\chi_r = \psi \circ \varphi_{i_{\max}(r)} \circ \dots \circ \varphi_0$ from \Cref{proposition: Pak map as tropical maps}. For all $i \in \{0, \dots, i_{\max}(r) \}$ and $\ell \in \diags(\lambda)$, we have that $\varphi_i(H_\ell^C) = H_\ell^C$ and $\psi(H_\ell^C) = H_\ell^{C \cup \{r\}}$. Hence $\chi_r: \mO_C(\lambda)^k_{\underline{d}} \to \mO_{C\cup\{r\}}(\lambda)^k_{\underline{d}}$ is a bijection.
\end{lemma}
\begin{proof}
By \cite{pak2001hook}, we have that $\chi_r$ admits an inverse given by:
\[ 
\chi_r^{-1}(x)_{(i,j)} = 
\begin{cases} 
x_{(i,j)}+ \max\{x_{(i-1,j)},x_{(i,j-1)}\} & \text{ if } (i,j) = r,\\
\max\{x_{(i-1,j)},x_{(i,j-1)}\} & \\
\quad\quad + \min\{x_{(i+1,j)}, x_{(i,j+1)}\} -x_{(i,j)} & \text{ if } i-j=\ell \text{ and } (i,j)< r,\\
x_{(i,j)} & \text{ otherwise. }
\end{cases}
\]
Hence, to show that $\chi_r: \mO_C(\lambda) \to \mO_{C\cup\{r\}}(\lambda)$ is a bijection, it suffices to prove that $\chi_r\left( \mO_C(\lambda) \right) = \mO_{C\cup \{r\}}(\lambda)$. Fix $x \in \mO_C(\lambda)$. We show that $\chi_r(x) \in \mO_{C \cup\{r\}}(\lambda)$ by showing that $\chi_r(x)$ satisfies all the defining inequalities of $\mO_{C \cup \{r\}}(\lambda)$:
\begin{itemize}
    \item If $i-j\neq \ell$ or both $i-j = \ell$ and $(i,j)>r$, then $0 \leq \chi_r(x)_{(i,j)} = x_{(i,j)} \leq 1$. Since $r$ is a corner of $\lambda \backslash C$, for every $(i,j) \leq r$ with $i-j=\ell$, it follows that $\max\{x_{(i-1,j)},x_{(i,j-1)}\} \leq x_{(i,j)} \leq \min\{x_{(i+1,j)},x_{(i,j+1)}\}$, hence $0 \leq \chi_r(x)_{(i,j)}\leq 1$.
    
    \item Let $(i_1,j_1) \leq (i_2,j_2)$ in $\lambda \backslash (C \cup \{r\})$. Since, for every $(i,j) \in \lambda \backslash (C \cup \{r\})$ with $i-j =\ell$ we have that $\max\{x_{(i-1,j)},x_{(i,j-1)}\} \leq \chi_r(x)_{(i,j)} \leq \min\{x_{(i+1,j)},x_{(i,j+1)}\}$, it follows that $\chi_r(x)_{(i_1,j_1)} \leq \chi_r(x)_{(i_2,j_2)}$.
    
    \item It remains to show that if $p \in \lambda\backslash (C\cup\{r\})$ and $q_1,\dots,q_n \in C\cup \{r\}$ with $p < q_1 < \dots < q_n$, then $\chi_r(x)_p + \chi_r(x)_{q_1} + \dots + \chi_r(x)_{q_n} \leq 1$. First assume that $q_i \neq r$ for every $i\in\{1,\dots,n\}$. If $p$ does not lie on the $\ell$th diagonal, then $\chi_r(x)_p + \chi_r(x)_{q_1} + \dots + \chi_r(x)_{q_n} = x_p + x_{q_1}+\dots+x_{q_n} \leq 1$. Otherwise, if $p$ lies on the $\ell$th diagonal, then $\chi_r(x)_p + \chi_r(x)_{q_1} + \dots + \chi_r(x)_{q_n} \leq x_r + x_{q_1}+\dots+x_{q_n} \leq 1$.
    
    Suppose that $r \in \{q_1,\dots,q_n\}$. It follows that $r = q_1$. Since $p < q_1$ and $x_p \leq \max\{ x_{(a-1,b)}, x_{(a,b-1)} \}$, we have  $\chi_r(x)_p+ \chi_r(x)_r = \chi_r(x)_p + x_{(a,b)} - \max\{ x_{(a-1,b)}, x_{(a,b-1)} \} \leq x_r$. Hence $\chi_r(x)_p + \chi_r(x)_{q_1} + \dots + \chi_r(x)_{q_n} \leq x_r + x_{q_2} + \dots + x_{q_n} \leq 1$.
\end{itemize}

So, we have shown that $\chi_r(\mO_C(\lambda)) \subseteq \mO_{C \cup\{r\}}(\lambda)$. Using the explicit formula for $\chi_r^{-1}$, the reverse inclusion $\chi_r^{-1}(\mO_{C\cup\{r\}}(\lambda)) \subseteq \mO_C(\lambda)$ follows similarly.

Recall the decomposition $\chi_r = \psi \circ \varphi_{i_{\max}(r)} \circ \dots \circ \varphi_0$ from \Cref{proposition: Pak map as tropical maps} and the defining equation of the hyperplane $H_\ell^C$ given by
\[ 
H^C_\ell: \quad \sum_{s \in \overline{S_\ell^C}} x_s
- \sum_{t \in \overline{T_\ell^C}} x_t
+ \sum_{u \in C \cap R_\ell} x_u = d_\ell.
\]
Consider the tropical map $\varphi_i = \varphi_{w_i, F_i}$. We have that $w_i = e_{\alpha} - e_{\beta}$ for some $\alpha$ and $\beta$ in $\lambda$ that lie on the same diagonal of $\lambda \backslash C$. Since the sets $\overline{S_\ell^C}$ and $\overline{T_\ell^C}$ are unions of diagonals of $\lambda \backslash C$, it follows that $\varphi_i(H_\ell^C) = H_\ell^C$ for each $i \in \{0,\dots, i_{\max}(r) \}$ and $\ell \in \diags(\lambda)$.

It remains to show that $\psi(H^C_\ell) = H^{C\cup\{r\}}_\ell$. Let $x \in H^C_\ell$, then we have
\begin{align*}
\sum_{\substack{i-j=\ell \\ (i,j)\leq r}} x_{(i,j)} = 
\psi(x)_{(a,b)} +\psi(x)_{(a-1,b)} +\psi(x)_{(a,b-1)} +
\sum_{\substack{i-j=\ell \\ (i,j)<r}} (-\psi(x)_{(i,j)} + \psi(x)_{(i-1,j)} +\psi(x)_{(i,j-1)}).
\end{align*}
%Note that there are four possible behaviours of $C$ near the corner $r$ of $\lambda\backslash C$:
Since $\psi(x)_{(i,j)}=x_{(i,j)}$ for every $(i,j)$ with $i-j\neq \ell$, using the previous equality, the following holds:
\begin{align*} 
d_\ell &= \sum_{s \in \overline{S_\ell^C}} x_s
  - \sum_{t \in \overline{T_\ell^C}} x_t
  + \sum_{u \in C \cap R_\ell} x_u \\
&= \sum_{\substack{s \in \overline{S_\ell^C}\\ i-j\neq \ell}} x_s
  - \sum_{t \in \overline{T_\ell^C}} x_t
  + \sum_{u \in C \cap R_\ell} x_u 
  + \sum_{\substack{i-j=\ell\\ (i,j)\leq r}} x_{(i,j)} \\
&= \sum_{\substack{s \in \overline{S_\ell^C}\\ i-j\neq \ell}} \psi(x)_s
- \sum_{t \in \overline{T_\ell^C}} \psi(x)_t
+ \sum_{u \in C \cap R_\ell} \psi(x)_u \\
&\quad + \psi(x)_{(a,b)} - \sum_{\substack{i-j=\ell\\(i,j)<r}} \psi(x)_{(i,j)}  +\sum_{\substack{i-j=\ell\\ (i,j)\leq r}} \psi(x)_{(i-1,j)} +\psi(x)_{(i,j-1)}
\end{align*}
Note that $r \in T_\ell^{C\cup\{r\}}$. There are four possible behaviours of $C$ near the corner $r$ of $\lambda \backslash C$ as shown below, where the highlighted boxes lie in $C$:

%\textcolor{blue}{[Add picture]}
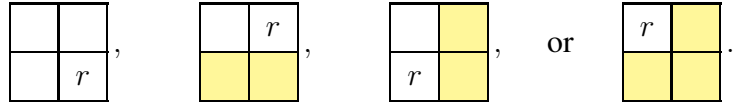
\[ \begin{ytableau}{} & \\ & r \end{ytableau} \, , \qquad \; \begin{ytableau}{} & r\\ *(yellow!50) & *(yellow!50) \end{ytableau} \, , \qquad \; \begin{ytableau}{} & *(yellow!50) \\ r & *(yellow!50) \end{ytableau} \, , \quad \text{ or } \quad \; \begin{ytableau}{} r & *(yellow!50) \\ *(yellow!50) & *(yellow!50) \end{ytableau}\, .
\]

In particular, since $r$ is a corner of $\lambda \backslash C$, we have that $(a-1,b)$ and $(a,b-1)$ are not in $S_\ell^C$. Moreover, $(a+1,b) \in T_\ell^C$ if and only if $(a,b-1) \not\in S_\ell^{C\cup\{r\}}$. Similarly, $(a,b+1) \in T_\ell^C$ if and only if $(a-1,b) \not\in S_\ell^{C\cup\{r\}}$. So, the previous expression can be written as:
\[\sum_{s \in \overline{S^{C\cup\{r\}}_\ell}} \psi(x)_s - \sum_{t \in \overline{T^{C \cup\{r\}}_\ell}} \psi(x)_t + \sum_{u \in (C\cup\{r_\ell\}) \cap R_\ell} \psi(x)_u = d_\ell \]
that is, $\psi(H^C_\ell) = H^{C \cup\{r\}}_\ell$.
\end{proof}

\begin{lemma}\label{lemma: tropical map is a mutation}
Fix Setup~\ref{setup: restricted chain-order polytope} and let $C \subsetneq \lambda$ be an up-set. Let $r = (a, b)$ be a corner of $\lambda \backslash C$ and recall the decomposition $\chi_r = \psi \circ \varphi_{i_{\max}(r)} \circ \dots \circ \varphi_0$ from \Cref{proposition: Pak map as tropical maps}. For every $i \in \{0,\dots,i_{\max}(r)\}$, the tropical map $\varphi_i= \varphi_{w_i,F_i}$ gives a combinatorial mutation of the polytope $(\varphi_{i-1}\circ \dots \circ \varphi_0) (\mO_C(\lambda)^k_{\underline{d}})$.
\end{lemma}
\begin{proof} 
By \Cref{lemma : chi_r bijection restricted polytopes}, for each $i \in \{0, \dots, i_{\max}(r)\}$ and $\ell \in \{1 - m_2, \dots, m_1 - 1 \}$, the hyperplane $H_\ell^C$ is invariant under $\varphi_i$. Hence, it is enough to prove that $\varphi_i$ gives a combinatorial mutation of the polytope $P_i := (\varphi_{i-1}\circ \dots \circ \varphi_0)(\mO_C(\lambda))$. 
% By convention, we take $P_0 := \mO_C(\lambda)$.

We introduce the following notation for readability:
\[
A = (a-i-1, b-i-1),\ 
B = (a-i-1, b-i),\ 
C = (a-i, b-i-1),\ 
D = (a-i, b-i).
\]
Recall that $F_i = \conv\{e_B, e_C\}$. We may assume that both $B$ and $C$ lie in $\lambda$, otherwise $\varphi_i$ is a unimodular map and the result follows immediately. So, the regions of linearity of the tropical map $\varphi_i$ are given by
\[
U_\ge := \{x \in \RR^\lambda \mid x_B \ge x_C\} 
\text{ and }
U_\le := \{x \in \RR^\lambda \mid x_B \le x_C\}.
\]
Similarly, we define $U_>$ and $U_<$ to be the interiors of $U_\ge$ and $U_\le$ respectively. Let $U_= := U_\ge \cap U_\le$ be the intersection of the regions of linearity.

% \color{blue}
% \begin{itemize}
%     \item Take a pair of points $p \in U_> \cap P_i$ and $q \in U_< \cap P_i$.
    
%     \item consider the four coordinates of $p$ and $q$: $A = (a-i,b-i)$, $B = (a-i-1, b-i)$, $C = (a-i, b-i-1)$ and $D = (a-i-1,b-i-1)$.
    
%     \item Form some new points by averaging over $p$ and $q$ so that we get a square that contains the line segment between $p$ and $q$.
    
%     \item This shows that the image of $\varphi_i^{-1}$ is contained in $P_i$.
    
%     \item Do the same for a pair of points $p \in U_> \cap \varphi_i(P_i)$ and $q \in U_< \cap \varphi_i(P_i)$.
% \end{itemize}
% \color{black}
% Starting with $\mO_C(\lambda)$, for each mutation, consider regions of linearity of the tropical map. Then give local conditions on which points of the polytope lie on each side. Use the rectangle trick: $a+b = c+d$ for points on each side of the hyperplane between the regions of linearity.

Let $p \in U_> \cap P_i$ and $q \in U_< \cap P_i$ be two points that lie in the interiors of the regions of linearity. To show that $P_{i+1} := \varphi_i(P_i)$ is convex, it suffices to show that the line segment between $\varphi_i(p)$ and $\varphi_i(q)$ is contained in $P_{i+1}$. Without loss of generality, we may assume that $p$ and $q$ are vertices of $P_i$. For each $i \in \{0, \dots, i_{\max}(r) \}$, we will show that the line segment $[\varphi_i(p), \varphi_i(q)]$ is contained in $P_{i+1}$. Moreover, we will show that $\varphi_i$ gives a bijection between the vertices of $P_i$ and $P_{i+1}$. 
% Since $\varphi_i$ is linear on $U_\ge$ and $U_\le$, it is enough to show that the point $v_0 := [\varphi_i(p), \varphi_i(q)] \cap U_=$ is contained in $\varphi_i(P_i)$. 
We proceed by induction on $i$. 

\smallskip

Fix $i = 0$. Since $p_B > p_C$, $q_B > q_C$, and $p$ and $q$ are vertices of $\mO_C(\lambda)$, we have that
\[
\begin{cases}
p_A = p_C = 0\\
p_B = p_D = 1
\end{cases}
\quad \text{and} \quad 
\begin{cases}
q_A = q_B = 0\\
q_C = q_D = 1.
\end{cases}
\]
Note that, by the definition of the chain-order polytope, we have that $p_s = q_s = 0$ for all $s \le A$.
% Define $v := \frac 12(p+q)$ to be the mid-point of the line-segment $[p, q]$, which lies in the space $U_=$.
We define the points $\alpha, \beta, \gamma, \delta \in \RR^\lambda$ as follows:
\begin{itemize}
    \item $\alpha_s = 0$ for all $s \le D$ and $\alpha_s = \min\{p_s, q_s\}$ for all $s \nleq D$, in particular $\alpha_A = \alpha_B = \alpha_C = \alpha_D = 0$,
    
    \item $\beta_s = 0$ for all $s < D$ and $\beta_s = \min\{p_s, q_s \}$ for all $s \nless D$, in particular $\beta_A = \beta_B = \beta_C = 0$ and $\beta_D = 1$,
    
    \item $\gamma_s = 0$ for all $s \le A$ and $\gamma_s = \max\{p_s, q_s \}$ for all $s \nleq A$, in particular $\gamma_A = 0$ and $\gamma_B = \gamma_C = \gamma_D = 1$,
    
    \item $\delta_s = 0$ for all $s < A$, $\delta_A = 1$, and $\delta_s = \max\{p_s, q_s\}$ for all $s \nleq A$, in particular $\delta_A = \delta_B = \delta_C = \delta_D = 1$.
\end{itemize}
It is straightforward to check that the points $\alpha, \beta, \gamma$ and $\delta$ lie in $\mO_C(\lambda)$. Observe that $\alpha, \beta, \gamma, \delta \in U_=$ and $p + q = \beta + \gamma$. In particular, the line segment $[p,q]$ is not an edge of $\mO_C(\lambda)$. By the definition of the tropical map $\varphi_0$, we have that $\varphi_0(p) + \varphi_0(q) = \varphi_0(\alpha) + \varphi_0(\delta)$. To see this, note that $\varphi_0$ fixes all coordinates except $A$ and $D$, which vary based only on the values of $B$ and $C$. Writing out the coordinates $A, B, C, D$ of these points we have:
\begin{align*}
    \varphi_0(p) + \varphi(q) 
        &=  \varphi_0 \begin{pmatrix} 0 & 1 \\ 0 & 1 \end{pmatrix} + 
            \varphi_0 \begin{pmatrix} 0 & 0 \\ 1 & 1 \end{pmatrix} = 
            \begin{pmatrix} 0 & 1 \\ 0 & 1 \end{pmatrix} +
            \begin{pmatrix} 0 & 0 \\ 1 & 1 \end{pmatrix} =
            \begin{pmatrix} 0 & 1 \\ 1 & 2 \end{pmatrix} \\
        &=  \begin{pmatrix} 0 & 0 \\ 0 & 0 \end{pmatrix} + 
            \begin{pmatrix} 0 & 1 \\ 1 & 2 \end{pmatrix} =
            \varphi_0 \begin{pmatrix} 0 & 0 \\ 0 & 0 \end{pmatrix} + 
            \varphi_0 \begin{pmatrix} 1 & 1 \\ 1 & 1 \end{pmatrix} =
            \varphi_0(\alpha) + \varphi(\delta)
\end{align*}
where order of the coordinates is given by $\begin{pmatrix} A & B \\ C & D \end{pmatrix}$. For the remaining coordinates, the equality follows immediately from the definition of $\alpha$ and $\delta$ above. Therefore, the line segment $[\varphi_0(p), \varphi(q)]$ is contained in the convex hull $\conv\{\varphi_0(p), \varphi_0(q), \varphi_0(\alpha), \varphi_0(\delta)\}$. Hence $P_1$ is convex. Observe that the line segment $[\varphi_0(p), \varphi(q)]$ is not an edge of $P_1$. 

\smallskip

We now show that $\varphi_0$ gives a bijection between the vertices of $\mO_C(\lambda)$ and $P_1$. Clearly, $\varphi_0$ is a bijection between the vertices of $P_1$ and $\mO_C(\lambda)$ that lie in the interiors $U_>$ and $U_<$. Let $v$ be any vertex of $\mO_C(\lambda)$ that lies in $U_=$. Assume by contradiction that $\varphi_0(v)$ is not a vertex of $P_1$, then $\varphi_0(v)$ lies in strict interior of an edge of $P_1$. Therefore, there exist vertices $p \in P_1 \cap U_>$ and $q \in P_0 \cap U_<$ such that $\varphi_0(v)$ lies on the line segment between them, which is an edge of $P_1$. However, we have already observed that all such line segments are not edges of $P_1$, a contradiction. Therefore $\varphi_0(v)$ is a vertex of $P_0$.

Now, let $v$ be a non-vertex of $\mO_C(\lambda)$ and assume by contradiction that $\varphi_0(v)$ is a vertex of $P_1$. Since $\varphi_0(v)$ is a vertex, it follows that $v$ lies on an edge of $\mO_C(\lambda)$ that intersects the interiors of both regions of linearity $U_> $ and $U_<$. Hence, there exist vertices $p \in U_> \cap \mO_C(\lambda)$ and $q \in U_< \cap \mO_C(\lambda)$ such that $v \in [p, q]$. However, we have already observed that there are no such edges, a contradiction. Therefore, $\varphi_0(v)$ is not a vertex of $P_1$. And so we have shown that $v$ is a vertex of $\mO_C(\lambda)$ if and only if $\varphi_0(v)$ is a vertex of $P_0$.

\smallskip

The inductive step, for $i > 0$, follows almost identically to the case $i = 0$. The only difference is that, for each $x \in \RR^\lambda$, we do not work with the value $x_{(a-i, b-i)}$. Instead, we use the value $x_{(a-i, b-i)} + \min\{x_{(a-i+1, b-i)}, x_{(a-i, b-i+1)}\}$.

\end{proof}

\begin{proof}[\bf Proof of \Cref{thm : restricted order and chain mutation equivalent}]
Fix a maximal flag of up-sets of $\lambda$: $\emptyset = C_0 \subsetneq C_1 \subsetneq C_2 \subsetneq \dots \subsetneq C_u = \lambda$. Hence, for each $i \in \{1, \dots, u\}$, we have that $C_i \backslash C_{i-1}$ is a singleton $\{r\}$ such that $r$ is a corner of $\lambda \backslash C_{i-1}$. By \Cref{lemma : chi_r bijection restricted polytopes}, the map $\chi_r$ is a bijection between the restricted chain-order polytopes $\chi_r : \mO_{C_{i-1}}(\lambda)_{\underline d}^k \rightarrow \mO_{C_{i}}(\lambda)_{\underline d}^k$. By \Cref{proposition: Pak map as tropical maps} we have that $\chi_r = \psi \circ \varphi_{i_{\max}(r)} \circ \dots \circ \varphi_0$ is the composition tropical maps $\varphi_i$ and a unimodular map $\psi$. By \Cref{lemma: tropical map is a mutation}, we have that for each $j \in \{0, \dots, i_{\max}(r)\}$, the tropical map $\varphi_j$ gives a combinatorial mutation of the polytope $\varphi_{j-1} \circ \dots \circ \varphi_0(\mO_{C_{i-1}}(\lambda)_{\underline d}^k)$. So $\mO_{C_{i-1}}(\lambda)_{\underline d}^k$ and $\mO_{C_{i}}(\lambda)_{\underline d}^k$ are mutation equivalent for each $i$. Hence $\mO(\lambda)_{\underline d}^k = \mO_{C_0}(\lambda)_{\underline d}^k$ and $\mC(\lambda)_{\underline d}^k = \mO_{C_u}(\lambda)_{\underline d}^k$ are mutation equivalent.
\end{proof}

\begin{remark}
The proof of \Cref{lemma: tropical map is a mutation} shows that all chain-order polytopes $\mO_C(\lambda)$ have the same number of vertices. However, this does not hold for the restricted polytopes. See \cite[Examples~2.1 and 2.3]{andersson2022restricted}.
\end{remark}

\section{Observations about period collapse}\label{sec: computations}

In this section we focus on the phenomenon of period collapse, as studied in \cite{andersson2022restricted}. We provide some small computations, straightforward corollaries of our main results, and further questions. We begin with the following observation.

\begin{prop}[{\cite[Theorem~2.1 and Lemma~2.2]{andersson2022restricted}}]\label{prop: rectangular period collapse}
If $\lambda$ is rectangular then the Ehrhart quasi-polynomial of $\mO(\lambda)_{\underline d}^k$ is a polynomial.
\end{prop}

\begin{proof}
The proof follows identically to \cite[Theorem~2.1 and Lemma~2.2]{andersson2022restricted} with a slight modification for the rectangular case. Explicitly, if $\lambda = [m_1] \times [m_2]$, then $\mO(\lambda)_{\underline d}^k$ is integrally equivalent to the Gelfand-Tsetlin polytope $\GT_{\alpha, \beta}$ of shape $\alpha$ and content $\beta$ where $\alpha = (k^{(m_1)}, 0^{(m_2)})$ and 
\begin{multline*}
\beta = (d_{m_1 - 1}, \,
d_{m_1 - 2} - d_{m_1 - 1}, \, d_{m_1 - 3} - d_{m_1 - 2}, \,
  \dots, \, d_{m_1 - m_2} - d_{m_1 - m_2 + 1}, \\
k + d_{m_1 - m_2 - 1} - d_{m_1 - m_2}, \, k + d_{m_1 - m_2 - 2} - d_{m_1 - m_2 - 1}, \,
  \dots, \, k + d_{-m_2 + 1} - d_{-m_2 + 2} 
).
\end{multline*}
The notation $k^{(m_1)}$ above means $k$ repeated $n$ times.
\end{proof}

By \Cref{prop: mutation equiv same Ehrhart}, combinatorial mutations preserve the Ehrhart (quasi)-polynomial, so we immediately obtain the following corollary of \Cref{thm : restricted order and chain mutation equivalent}.

\begin{corollary}
If $\lambda$ is a rectangular Young diagram then the Ehrhart quasi-polynomial of the restricted chain-order polytope $\mO_C(\lambda)_{\underline d}^k$ is a polynomial, for any $\underline d$, $k$ and up-set $C \subseteq \lambda$.
\end{corollary}

We observe that some non-rectangular Young diagrams give rise to restricted chain-order polytopes with period collapse. For the next proposition, we require the following definition. Let $\lambda$ be a Young diagram and fix a pair of adjacent diagonals $\ell$ and $\ell+1$ where $\ell \ge 0$ (resp. $\ell$ and $\ell - 1$ where $\ell \le 0$). Assume the diagonals have the same length. We define \emph{the Young diagram $\lambda \backslash \ell$ with diagonal $\ell$ removed} as follows:
\[
\lambda \backslash \ell := 
\begin{cases}
\{(i,j) \in \lambda \mid i - j < \ell \} \cup
\{(i - 1,j) \mid (i,j) \in \lambda \text{ and } i - j > \ell \} & \text{if }  \ell \ge 0,\\
\{(i,j) \in \lambda \mid i - j > \ell \} \cup
\{(i,j-1) \mid (i,j) \in \lambda \text{ and } i - j < \ell \} & \text{resp. if } \ell \le 0.
\end{cases}
\]
Given a vector $\underline d$, as in \Cref{setup: restricted chain-order polytope} for the Young diagram $\lambda \subseteq [m_1] \times [m_2]$, we also define $\underline d \backslash \ell = (d_{m_1 - 1}, \dots, d_{\ell + 1}, d_{\ell - 1}, \dots, d_{1 - m_2})$. We think of $\underline d \backslash \ell$ as the corresponding vector for the Young diagram $\lambda     \backslash \ell$. See \Cref{example: equal to 3x3 square}.

\begin{prop}\label{prop: delete identical diagonals}
If $\lambda$ contains two adjacent diagonals of the same length, say $\ell$ and $\ell+1$ with $\ell \ge 0$ (resp. $\ell$ and $\ell - 1$ with $\ell \le 0$), and $d_\ell = d_{\ell+1}$ (resp. $d_{\ell} = d_{\ell - 1}$) then the polytopes $\mO(\lambda)_{\underline d}^k$ and $\mO(\lambda \backslash \ell)_{\underline d \backslash \ell}^k$ are integrally equivalent. Also if $d_{\ell+1} < d_\ell$ (resp. $d_{\ell - 1} < d_{\ell}$) then $\mO(\lambda)_{\underline d}^k$ is empty. 
\end{prop}

\begin{proof}
Assume $\ell \ge 0$ and take any point $x \in \mO(\lambda)_{\underline d}^k$. Write $\alpha_1, \dots, \alpha_t$ for the boxes along the $\ell$th diagonal in $\lambda$ and $\beta_1, \dots, \beta_t$ for the coordinates along the $(\ell + 1)$th diagonal. By definition of the restricted order polytope, we have that $x_{\alpha_1} + \dots + x_{\alpha_t} = d_\ell = d_{\ell + 1} = x_{\beta_1} + \dots + x_{\beta_t}$ and for each $i \in [t]$ we have $x_{\alpha_i} \le x_{\beta_i}$. It follows immediately that $x_{\alpha_i} = x_{\beta_i}$ for all $i \in [t]$. Therefore the projection from $\mO(\lambda)_{\underline d}^k$ to $\mO(\lambda \backslash \ell)_{\underline d \backslash \ell}^k$ which removes the $\ell$th diagonal of $\lambda$ gives an integral equivalence of polytopes. 

Since $x_{\alpha_i} \le x_{\beta_i}$ for all $i \in [t]$, it follows that $d_\ell \le d_{\ell + 1}$. So, if $d_{\ell + 1} < d_{\ell}$ then we have that $\mO(\lambda)_{\underline d}^k$ is empty. A similar argument proves the result when $\ell \le 0$.
\end{proof}

So for any Young diagram $\lambda$, we may apply Propositions~\ref{prop: rectangular period collapse} and \ref{prop: delete identical diagonals} to construct restricted order polytopes that exhibit period collapse. We obtain the following corollary of \Cref{thm : restricted order and chain mutation equivalent}.

\begin{corollary}
Let $\underline d$ be the vector given by $(\underline d)_\ell = |\{\text{boxes of the } \ell \text{th diagonal of }\lambda\}|$. Then the Ehrhart quasi-polynomial of $\mO_C(\lambda)_{\underline d}^k$ is a polynomial.
\end{corollary}

\begin{remark}
We note that \Cref{prop: rectangular period collapse} does not immediately generalise to the non-rectangular cases. Unless the vector $\underline d$ contains repeated entries, as in \Cref{prop: delete identical diagonals}, it is not clear whether restricted order polytopes $\mO(\lambda)_{\underline d}^k$ for non-rectangular posets $\lambda$ are related to the Gelfand-Tsetlin polytope $\GT_{\alpha, \beta}$. 
Restricted order polytopes can be realised as the intersection of a GT-polytope with certain hyperplanes. However, it is not clear why such an intersection would give a polytope with period collapse.
\end{remark}

\begin{example}\label{example: equal to 3x3 square}
Consider the partition of $11$ given by $(4,4,3)$ and let $\lambda$ be the corresponding Young diagram. Let $k = 2$ and $\underline d = (1,2,3,2,2,1)$. Then the vertices of restricted order $\mO(\lambda)_{\underline d}^k$ are
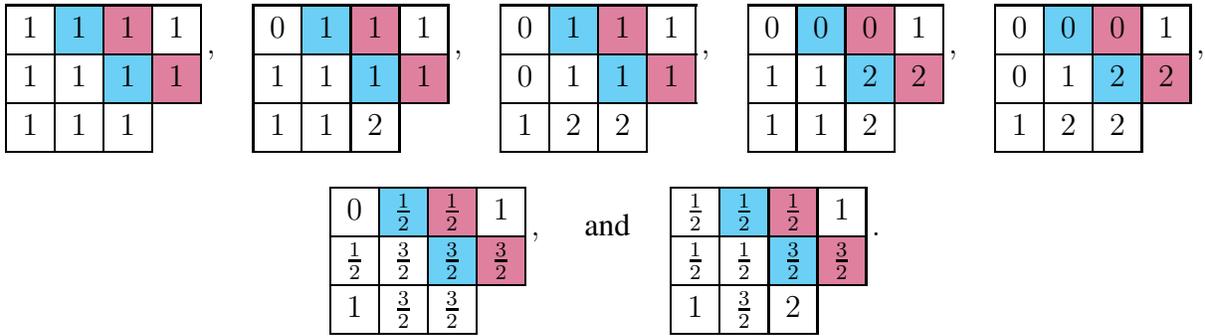
\[
\begin{ytableau}
1 & *(cyan!50) 1 & *(purple!50) 1 & 1 \\
1 & 1 & *(cyan!50) 1 & *(purple!50) 1 \\
1 & 1 & 1
\end{ytableau}\, , \quad
\begin{ytableau}
0 & *(cyan!50) 1 & *(purple!50) 1 & 1 \\
1 & 1 & *(cyan!50) 1 & *(purple!50) 1 \\
1 & 1 & 2
\end{ytableau}\, , \quad
\begin{ytableau}
0 & *(cyan!50) 1 & *(purple!50) 1 & 1 \\
0 & 1 & *(cyan!50) 1 & *(purple!50) 1 \\
1 & 2 & 2
\end{ytableau}\, , \quad
\begin{ytableau}
0 & *(cyan!50) 0 & *(purple!50) 0 & 1 \\
1 & 1 & *(cyan!50) 2 & *(purple!50) 2 \\
1 & 1 & 2
\end{ytableau}\, , \quad
\begin{ytableau}
0 & *(cyan!50) 0 & *(purple!50) 0 & 1 \\
0 & 1 & *(cyan!50) 2 & *(purple!50) 2 \\
1 & 2 & 2
\end{ytableau}\, ,
\]
\[
\begin{ytableau}
0 & *(cyan!50) \frac 12 & *(purple!50) \frac 12 & 1 \\
\frac 12 & \frac 32 & *(cyan!50) \frac 32 & *(purple!50) \frac 32 \\
1 & \frac 32 & \frac 32
\end{ytableau}\, , \quad \text{ and } \quad
\begin{ytableau}
\frac 12 & *(cyan!50) \frac 12 & *(purple!50) \frac 12 & 1 \\
\frac 12 & \frac 12 & *(cyan!50) \frac 32 & *(purple!50) \frac 32 \\
1 & \frac 32 & 2
\end{ytableau}\, .
\]
The shaded adjacent diagonals $\ell = -1$ (in blue) and $\ell -1 = -2$ (in purple) have a same values since $d_{-1} = d_{-2} = 2$. Let $\lambda' = \lambda \backslash \ell$ be the square Young diagram given by the partition $(3,3,3)$ and $\underline d' = \underline d \backslash \ell = (1,2,3,2,1)$. By \Cref{prop: delete identical diagonals}, the vertices of the restricted order polytope $\mO(\lambda')_{\underline d'}^k$ can be obtained from the above vertices by applying the projection:
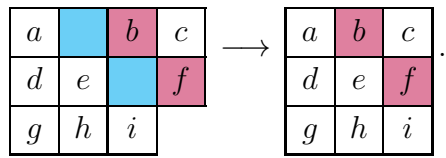
\[
\begin{ytableau}
a & *(cyan!50) & *(purple!50) b & c\\
d & e & *(cyan!50) & *(purple!50) f \\
g & h & i
\end{ytableau} 
\, \longrightarrow \,
\begin{ytableau}
a & *(purple!50) b & c\\
d & e & *(purple!50) f \\
g & h & i
\end{ytableau}\, .
\]
This projection gives an integral equivalence between the two restricted order polytopes. The Ehrhart polynomial of this polytope is given by:
\[
\frac {1}{12} t^4 + \frac 12 t^3 + \frac{17}{12}t^2 + 2t + 1
\]
and its $h^*$-vector is $(1,0,1,0,0)$.
\end{example}

\begin{example}\label{example: first non-square case}
We give an example of a restricted order polytope for some non-rectangular Young diagram $\lambda$ that does not satisfy the assumptions Propositions~\ref{prop: mutation equiv same Ehrhart} or \ref{prop: delete identical diagonals}. Consider the partition $\lambda = (4,4,3)$, $k = 3$ and $\underline d = (1,2,3,2,3,1)$. We index the coordinates of $\RR^\lambda$ as follows:
\[
\begin{ytableau}
1 & 2 & 3 & 4 \\
5 & 6 & 7 & 8 \\
9 & 10 & 11
\end{ytableau}\, \text{ and let } V = 
\left(
\begin{smallmatrix}
        1&0&0&0&0&0&0&0&0&0&0\\
        1&1&0&0&1&0&0&0&0&\frac{1}{2}&\frac{1}{2}\\
        1&1&1&0&1&1&0&1&0&1&\frac{1}{2}\\
        1&1&1&1&1&1&1&1&1&1&1\\
        1&1&1&1&0&0&0&0&0&\frac{1}{2}&\frac{1}{2}\\
        1&1&1&1&1&1&1&0&0&\frac{3}{2}&\frac{3}{2}\\
        1&1&2&2&1&2&2&2&2&\frac{3}{2}&\frac{3}{2}\\
        2&2&2&3&2&2&3&2&3&2&\frac{5}{2}\\
        1&1&1&1&1&1&1&1&1&1&1\\
        1&1&1&1&2&2&2&2&2&\frac{3}{2}&\frac{3}{2}\\
        1&2&2&2&2&2&2&3&3&\frac{3}{2}&\frac{3}{2}
\end{smallmatrix}
\right)
\in (\RR^{\lambda})^{11}
\]
The vertices of the restricted chain-order polytope $\mO(\lambda)_{\underline d}^k$ are the columns of $V$. The Ehrhart quasi-polynomial of this polytope is the polynomial given by
\[
\frac{7}{120}t^{5}+\frac{1}{2}t^{4}+\frac{41}{24}t^{3}+3t^{2}+\frac{41}{15}t+1.
\]
The $h^*$-vector for this polytope is $(1,3,3,0,0,0)$.
\end{example}

\begin{example}\label{example: larger non-square case}
Let $\lambda$ be the Young diagram associated to the partition $(4,4,3,2)$. Let $k = 3$ and $\underline d = (1,3,2,3,2,3,1)$. Label the coordinates of $\RR^\lambda$ as follows:
\[
\begin{ytableau}
1 & 2 & 3 & 4 \\
5 & 6 & 7 & 8 \\
9 & 10 & 11 \\
12 & 13
\end{ytableau}\, \text{ and let } V = \left(\begin{smallmatrix}
       1&0&0&0&0&0&0&0&0&0&0&0&0&0&0&0&0&0\\
       1&1&0&0&1&0&0&0&0&1&0&0&0&0&\frac{1}{2}&\frac{1}{2}&\frac{1}{2}&\frac{1}{2}\\
       1&1&1&0&1&1&0&1&0&1&1&0&1&0&1&\frac{1}{2}&1&\frac{1}{2}\\
       1&1&1&1&1&1&1&1&1&1&1&1&1&1&1&1&1&1\\
       1&1&1&1&0&0&0&0&0&0&0&0&0&0&\frac{1}{2}&\frac{1}{2}&\frac{1}{2}&\frac{1}{2}\\
       1&1&1&1&1&1&1&0&0&1&1&1&0&0&\frac{3}{2}&\frac{3}{2}&\frac{3}{2}&\frac{3}{2}\\
       1&1&2&2&1&2&2&2&2&1&2&2&2&2&\frac{3}{2}&\frac{3}{2}&\frac{3}{2}&\frac{3}{2}\\
       2&2&2&3&2&2&3&2&3&2&2&3&2&3&2&\frac{5}{2}&2&\frac{5}{2}\\
       1&1&1&1&1&1&1&1&1&0&0&0&0&0&1&1&\frac{1}{2}&\frac{1}{2}\\
       1&1&1&1&2&2&2&2&2&2&2&2&2&2&\frac{3}{2}&\frac{3}{2}&\frac{3}{2}&\frac{3}{2}\\
       1&2&2&2&2&2&2&3&3&2&2&2&3&3&\frac{3}{2}&\frac{3}{2}&\frac{3}{2}&\frac{3}{2}\\
       1&1&1&1&1&1&1&1&1&1&1&1&1&1&1&1&1&1\\
       2&2&2&2&2&2&2&2&2&3&3&3&3&3&2&2&\frac{5}{2}&\frac{5}{2}
\end{smallmatrix}\right)
\in (\RR^\lambda)^{18}
\]
The restricted order polytope $\mO(\lambda)_{\underline d}^k$ has $18$ vertices which are given by the columns of $V$.
Similarly, to the previous examples, the polytope exhibits period collapse. Its Ehrhart polynomial is
\[
\frac{7}{240}t^{6}+\frac{77}{240}t^{5}+\frac{23}{16}t^{4}+\frac{163}{48}t^{3}+\frac{68}{15}t^{2}+\frac{197}{60}t+1
\]
and its $h^*$-vector is $(1,7,11,2,0,0,0)$.
\end{example}

\begin{question}
Is the Ehrhart quasi-polynomial of the restricted chain-order polytope $\mO_C(\lambda)_{\underline d}^k$ a polynomial for any up-set $C \subseteq \lambda$, vector $\underline d$ and positive integer $k$?
\end{question}

By Theorem~\ref{thm : restricted order and chain mutation equivalent}, it suffices to consider only the restricted order polytopes. Also, using the combinatorial mutations defined in \Cref{sec: main results}, the question is simultaneously answered for all other intermediate polytopes that appear in the sequence of combinatorial mutations between restricted chain-order polytopes. 

\begin{question}
What is the degree of the $h^*$-polynomial of a restricted chain-order polytope?
\end{question}

The examples above seem to suggest that the degree of the $h^*$-polynomial is bounded above by half the dimension of the polytope. Given a Young diagram with $n$ boxes and $\ell$ diagonals, for sufficiently generic $k$ and $\underline d$, the dimension of the the restricted order polytope $\mO(\lambda)_{\underline d}^k$ is $n - \ell$.

\begin{question}
Which non-lattice restricted order polytopes are mutation equivalent to lattice polytopes?
\end{question}

Consider the restricted order polytope in \Cref{example: equal to 3x3 square}. It turns out that all intermediate polytopes, which appear in the sequence of mutations constructed in \Cref{sec: main results}, are non-lattice polytopes. For other non-lattice restricted order polytopes, we ask whether any of the intermediate polytopes are lattice polytopes. We note that if a restricted order polytope is mutation equivalent to a lattice polytope, then this gives an alternative proof that the Ehrhart quasi-polynomial is a polynomial.

\bibliographystyle{abbrv} 
\bibliography{bibfile.bib}

\begin{thebibliography}{10}

\bibitem{Akhtar2012}
M.~{Akhtar}, T.~{Coates}, S.~{Galkin}, and A.~M. {Kasprzyk}.
\newblock Minkowski polynomials and mutations.
\newblock {\em SIGMA. Symmetry, Integrability and Geometry: Methods and
  Applications}, 8:paper 094, 17, 2012.

\bibitem{andersson2022restricted}
P.~Alexandersson, S.~Hopkins, and G.~Zaimi.
\newblock Restricted {B}irkhoff polytopes and {E}hrhart period collapse.
\newblock {\em arXiv preprint arXiv:2206.02276}, 2022.

\bibitem{ardilla2011GTandFFLVPolytopes}
F.~Ardila, T.~Bliem, and D.~Salazar.
\newblock Gelfand–{T}setlin polytopes and
  {F}eigin–{F}ourier–{L}ittelmann–{V}inberg polytopes as marked poset
  polytopes.
\newblock {\em Journal of Combinatorial Theory, Series A}, 118(8):2454--2462,
  2011.

\bibitem{beck2007computing}
M.~Beck and S.~Robins.
\newblock {\em Computing the continuous discretely}.
\newblock Undergraduate Texts in Mathematics. Springer New York, 2007.

\bibitem{BSW_2008}
M.~Beck, S.~V. Sam, and K.~M. Woods.
\newblock Maximal periods of ({E}hrhart) quasi-polynomials.
\newblock {\em J. Combin. Theory Ser. A}, 115(3):517--525, 2008.

\bibitem{clarke2021mutationsBlockDiagonal}
O.~Clarke, A.~Higashitani, and F.~Mohammadi.
\newblock Combinatorial mutations and block diagonal polytopes.
\newblock {\em Collectanea Mathematica}, pages 1--31, 2021.

\bibitem{clarke2022mutationsGTandFFLV}
O.~Clarke, A.~Higashitani, and F.~Mohammadi.
\newblock Combinatorial mutations of {G}elfand-{T}setlin polytopes,
  {F}eigin-{F}ourier-{L}ittelmann-{V}inberg polytopes, and block diagonal
  matching field polytopes.
\newblock {\em arXiv preprint arXiv:2208.04521}, 2022.

\bibitem{clarke2022toric}
O.~Clarke, F.~Mohammadi, and F.~Zaffalon.
\newblock Toric degenerations of partial flag varieties and combinatorial
  mutations of matching field polytopes.
\newblock {\em arXiv preprint arXiv:2206.13975}, 2022.

\bibitem{escobar2019wall}
L.~Escobar and M.~Harada.
\newblock Wall-crossing for {N}ewton-{O}kounkov bodies and the tropical
  {G}rassmannian.
\newblock {\em International Mathematics Research Notices, rnaa230}, 2020.

\bibitem{FF_2016}
X.~Fang and G.~Fourier.
\newblock Marked chain-order polytopes.
\newblock {\em European J. Combin.}, 58:267--282, 2016.

\bibitem{Gelfand1950finite}
I.~M. Gelfand and M.~L. Tsetlin.
\newblock {Finite-dimensional representations of the group of unimodular
  matrices}.
\newblock {\em Dokl. Akad. Nauk SSSR}, 71(5):825--828, 1950.

\bibitem{HM_2008}
C.~Haase and T.~B. McAllister.
\newblock Quasi-period collapse and {${\rm GL}_n(\Bbb Z)$}-scissors congruence
  in rational polytopes.
\newblock In {\em Integer points in polyhedra---geometry, number theory,
  representation theory, algebra, optimization, statistics}, volume 452 of {\em
  Contemp. Math.}, pages 115--122. Amer. Math. Soc., Providence, RI, 2008.

\bibitem{higashitani2020two}
A.~Higashitani.
\newblock Two poset polytopes are mutation-equivalent.
\newblock {\em arXiv:2002.01364}, 2020.

\bibitem{kirillov2001ubiquity}
A.~N. Kirillov.
\newblock Ubiquity of {K}ostka polynomials.
\newblock {\em Physics and Combinatorics}, pages 85--200, 2001.

\bibitem{MW_2005}
T.~B. McAllister and K.~M. Woods.
\newblock The minimum period of the {E}hrhart quasi-polynomial of a rational
  polytope.
\newblock {\em J. Combin. Theory Ser. A}, 109(2):345--352, 2005.

\bibitem{pak2001hook}
I.~Pak.
\newblock Hook length formula and geometric combinatorics.
\newblock {\em S{\'e}minaire Lotharingien de Combinatoire}, 46:B46f, 13 p.,
  2001.

\bibitem{schensted_1961}
C.~Schensted.
\newblock Longest increasing and decreasing subsequences.
\newblock {\em Canadian Journal of Mathematics}, 13:179–191, 1961.

\bibitem{MR824105}
R.~P. Stanley.
\newblock Two poset polytopes.
\newblock {\em Discrete Comput. Geom.}, 1(1):9--23, 1986.

\end{thebibliography}

\noindent 
\textbf{Author's addresses.}

\noindent 
Department of Pure and Applied Mathematics, 
Osaka University, Suita, Osaka 565-0871, Japan\\
E-mail address: {\tt oliver.clarke.crgs@gmail.com}

\medskip\noindent
Department of Pure and Applied Mathematics,
Osaka University, Suita, Osaka 565-0871, Japan\\
E-mail address:  {\tt higashitani@ist.osaka-u.ac.jp}

\medskip  \noindent
Department of Mathematics, KU Leuven, Celestijnenlaan 200B, B-3001 Leuven, Belgium
\\ E-mail address: {\tt francesca.zaffalon@kuleuven.be}

\end{document}